\documentclass[11pt,mathserif]{amsart}
\usepackage{microtype}
\usepackage[utf8]{inputenc}
\usepackage{lmodern}
\usepackage{amsthm, amssymb, bm}

\usepackage{geometry}
\usepackage{marginnote}
\usepackage{amsaddr}
\usepackage{eulervm}

\usepackage[dvipsnames]{xcolor} % si no ponés [dvipsnames] se pudre todo con algunos colores

\usepackage{hyperref}
\usepackage{cleveref}
\hypersetup{
	linkcolor  = violet,
	citecolor  = YellowOrange,
	urlcolor   = Aquamarine,
	colorlinks = true,
}

\RequirePackage{multicol} % permite que algo ocupe varias columnas
\RequirePackage{multirow}% permite que algo ocupe varias filas 

\newtheorem{theorem}{Theorem}[section]
\newtheorem{theorem*}{Theorem}%[section]

\newtheorem{lemma}[theorem]{Lemma}
\newtheorem{proposition}[theorem]{Proposition}

\numberwithin{equation}{section}

\def\eint#1#2{[\hspace{-3pt}[#1\,,\,#2]\hspace{-2.6pt}]}

\usepackage{xcolor}

\title{Binary cyclotomic polynomials:  representation via words and algorithms}
\thanks{Supported by Project UNGS 30/3307, Project Stic AmSud 20STIC-06 and LIA SINFIN}
\thanks{The final authenticated publication is available at \url{https://doi.org/10.1007/978-3-030-85088-3_6}}

\author{Antonio Cafure - Eda Cesaratto}
\address{Instituto de Desarrollo Humano, UNGS; CONICET}
\email{acafure@campus.ungs.edu.ar, ecesaratto@campus.ungs.edu.ar}


\begin{document}

\maketitle
\begin{center}
 \sc Dedicated to the memory of Adri\'an Tato \'Alvarez    
\end{center}

\begin{abstract}
Cyclotomic polynomials are  basic objects in  Number Theory. Their properties  depend on the number of  distinct primes that intervene in the factorization of their order,  and the binary case is  thus the first nontrivial case. This paper sees the vector of coefficients of the polynomial as a word on a ternary alphabet $\{-1,0 ,+1\}$. It designs an efficient algorithm that  computes   a compact representation of this  word.  This algorithm   is of  linear time  with respect to the size of the output, and, thus, optimal.  
   {This approach allows to   recover known properties of coefficients of binary cyclotomic polynomials,  and extends to the case of  polynomials associated with numerical semi-groups of dimension 2. }
\end{abstract}

\section{Introduction}
\noindent
{\bf \em General description.} 
Cyclotomic polynomials are   defined as the irreducible factors in $\mathbb{Q}[x]$ of the polynomial $x^n-1$. The identity
\begin{equation} \label{eq:tn-1 equal to product of cyclotomic} x^n-1 = \prod_{d |n}  \Phi_d(x)\, 
\end{equation}
holds between the $d$-th cyclotomic polynomials $\Phi_d$.  The polynomial $\Phi_{d} (x)$ {equals} 
the product $\prod(x- \alpha)$ where $\alpha$  ranges over the primitive  $d$-th roots of unity.   It  has integer coefficients, {is self-reciprocal}, and irreducible over ${\mathbb Q}[x]$.  Its  degree is $\varphi(d)$, where $\varphi$ stands for Euler's totient function. 
The following  holds,  for any prime $p$, 
\begin{equation}\label{eq:Phipn from n}
 \begin{split}
  \Phi_{pn}(x)  & =   \Phi_{n}(x^{p})  
 \quad \mbox{if $p$ divides $n$, } 
  \\[1ex]	
  \Phi_{n}(x) \Phi_{pn}(x)  & =  \Phi_{n}(x^{p}) \quad \mbox{if $p$ does not divide $n$\, .}
  \end{split}
\end{equation}
 It is thus  enough to compute $\Phi_{d}$ when $d$ is squarefree.
Due to   {the equality}
\vskip 0.1 cm 
\centerline{  
$\Phi_{p} (x)  = x^{p-1} + \cdots + x + 1$,  }
\vskip 0.1 cm
\noindent
that holds  when  $d = p$ is prime,
the first nontrivial case occurs when $d$ is the product of two distinct primes $p$ and $q$, and defines what is called a \emph{binary} cyclotomic polynomial.

\smallskip
A cyclotomic polynomial  $\Phi_{d}$ is  usually  given by  its  dense representation, described  by the vector of  its coefficients in $\mathbb{Q}^{\varphi(d) + 1}$.  Many  properties  of interest  hold  on  the dense representation of a  binary cyclotomic polynomial $\Phi_{pq}$ (see \cite{HongLeePark12,LaLe99,Moree14}), for instance: 
\\$(a)$ the  polynomial $\Phi_{pq}$ has all its coefficients in the ternary alphabet $\{-1,0,+1\}$; \\ $(b)$ its nonzero coefficients alternate their signs: i.e. after a 1 follows a $-1 $ .
\\$(c)$ for $p < q$,  the maximum number of consecutive zeros in the vector of coefficient  %called the maximum gap,   
equals $p- 2$. 

\smallskip
Many algorithms  are designed for computing these polynomials, most of them being  based on identities of type \eqref{eq:tn-1 equal to product of cyclotomic} and \eqref{eq:Phipn from n}.   Around 2010, Arnold and Monagan  wished  to study  the coefficients of cyclotomic polynomials in an experimental way.   
 They  thus designed  efficient algorithms  that compute cyclotomic polynomials of very large order.  Their work \cite{ArMo11}  is the main reference on this subject.

\medskip 
\noindent
{\bf \em Our approach.} 
We are concerned  here with the representation and the computation of binary cyclotomic polynomials $\Phi_{pq}$. We let $m = \varphi(pq)= (p-1)(q-1)$, and 
we  consider  the {\em cyclotomic  vector} $\boldsymbol{a}_{pq}$, (that is moreover palindromic) formed with  coefficients of $\Phi_{pq}$,   from two points of view: 
\\
  $(i)$ first, classically, as a vector   of ${\mathbb Q}^{m + 1}$,   \\$(ii)$  but also as a word  of length $m + 1$  on  the ternary alphabet ${\mathcal A} = \{-1,0,  +1\}$.
  
  \smallskip
  \noindent
With the first point of view, we use classical arithmetical operations on vectors. With the second point of view, we use operations on words as cyclic permutations, concatenations and fractional powers.
In fact, we  use  {\em both} points of view  and then all the operations of these {\em  two} types. We have just to check that  the   sum   of   two words of $\mathcal{A}^{m+1}$   always belongs to $\mathcal{A}^{m+1}$.  We will see  in Lemma 1 that  this is  indeed the case   in our context:   the ``bad'' cases $(+1) + (+1)$ or $(-1) + (-1)$ never occur in the addition of two symbols of ${\mathcal A}$, and the sum remains ``internal". 

\smallskip
We  design an algorithm, called {BCW},   that {\em only} uses simple operations on words (concatenation, shift, fractional power, internal addition).  It takes as an input  a pair $(p, q)$ of two prime numbers $p, q$  with $p<q$, together with  {the quotient and the remainder of the} division of $q$ by $p$; this input is   thus  of size $\Theta (\log q)$. The algorithm outputs the   cyclotomic word $\boldsymbol{a}_{pq}$ of size $\Theta(pq)$; it   performs   a number $\Theta(pq)$  of operations on symbols of ${\mathcal A}$. The complexity of the algorithm is thus linear with respect to the size of the output, and thus optimal in this sense.  
The {\em proof} of the algorithm is  based on  arithmetical operations on polynomials, of type \eqref{eq:tn-1 equal to product of cyclotomic} or \eqref{eq:Phipn from n}, that are further  {\em transfered} into operations on words; however,  the algorithm  itself {\em does not perform} any  polynomial multiplication or  division.

\smallskip
To the best of our knowledge,   our approach, based on  combinatorics of words,  is quite novel inside the domain of cyclotomic polynomials. The compact representation of binary cyclotomic polynomials  described in Theorem 1  appears to be new and the algorithm {BCW}, whose complexity is $\Theta(pq)$, is more efficient that the already existing algorithms, whose complexity is $O(pq) E(p, q)$, where $E(p, q)$ is polynomial in $\log q$ (see Section \ref{sec:algo}).

\section{Statement of the main results }
We first define the main operations on words that will be used, then we state  our two main results:   the first one (Theorem 1) describes  a compact representation of the cyclotomic word, whereas the second result (Theorem 2) describes the algorithm -- so  called the \emph{binary cyclotomic word algorithm} (BCW algorithm for short)--  that is used to obtain it. Theorem 1 will be proven in Section \ref{section:binary cyclotomic polynomials} and Theorem 2 in Section \ref{sec:algo}.

\medskip
\noindent 
{\bf \em Operations on words.} 
We first define  the operations on words that we will use along the paper: concatenation and fractional power, cyclic permutation, addition.  

\smallskip\noindent 
{\bf \em Concatenation and fractional power.} Given two words $\boldsymbol{u}=u_{0}u_{1}\dots u_{k-1}$ and $\boldsymbol{v}=v_{0}v_{2}\dots v_{\ell-1}$ from the alphabet $\mathcal{A}$,  its  \emph{concatenation} is denoted by $\boldsymbol{u}\cdot\boldsymbol{v}$. For any  $s \in \mathbb{N}$, the  \emph{$s$-power} of a word is denoted by $\boldsymbol{v}^{s}$, and $\boldsymbol{v}^{0} = \epsilon$ is the empty word.
\\
Let $k$ and $p$ be positive integers. For any  $\boldsymbol{v}\in \mathcal{A}^{p}$,  the 
\emph{fractional power}  $\boldsymbol{v}^{k/p}$ of $\boldsymbol{v}$ is the element of $\mathcal{A}^{k}$ defined as follows:
    \\ $(a)$  For $k < p$, then $\boldsymbol{v}^{k/p}$ is the \emph{truncation} of the word $\boldsymbol{v}$ to its prefix of length $k$.
    \\$(b)$ For $k \geq p $, then $           \boldsymbol{v}^{k/p} = \boldsymbol{v}^{\lfloor k/p \rfloor}\cdot\boldsymbol{v}^{(k \!\!\pmod p)/p}$.

\smallskip\noindent
{\bf \em Circular permutation. } The \emph{left circular permutation} $\sigma$ 
is the  mapping $\sigma: \mathcal{A}^{p} \rightarrow \mathcal{A}^{p}$ defined as
   \begin{align}\label{eq: def of circular permutation}
  \hbox {if  $ \boldsymbol{v} = v_{0}\,  v_{1}\,v_2  \cdots \, v_{p-2}\, v_{p-1}$, then} \quad  \sigma(\boldsymbol{v})
  = v_{1}\, v_{2}  \cdots v_{p-2}\,  v_{p-1}\, v_{0}.
   \end{align}
    The mapping 
  $\sigma$ has order $p$: it satisfies for $s\in {\mathbb Z}$,
  
  \vskip 0.1cm
  \centerline{$
  \sigma^s = \sigma^{s {\, \rm mod\, }  p}, \qquad  \sigma^p = {\rm Id} \, .$}

\noindent
{\bf \em Addition. } The paper is based on a transfer  between algebra and combinatorics on words and leads to define addition between words;  
 we begin  with the usual  sum between two vectors of ${\mathbb Z}^k$, component by component: the sum  $ \boldsymbol{u}+\boldsymbol{v}$ between two words  $\boldsymbol{u}=u_{0}u_{1}\dots u_{k-1}$ and $\boldsymbol{v}=v_{0}v_{2}\dots v_{k-1}$  is

\vskip 0.1 cm
\centerline{$ 
\boldsymbol{u}+\boldsymbol{v}=(u_{0}+v_{0})(u_{1}+v_{1})\cdots (u_{k-1}+v_{k-1}).$}
\vskip 0.1 cm
\noindent Then,  the sum  of two words from $\mathcal{A}$ may have symbols not in $\mathcal{A}$, due to the two bad  cases   $[+1 + (+ 1)]$ and $[ (- 1) + (-1)]$. However, we will prove in Lemma \ref{lemma:suma de los di no se sale de A} that this  will never occur in our framework.
 Then, an addition between two symbols of ${\mathcal A}$   will be always   an \emph{internal} operation,   with  a sum  that stays inside $\mathcal{A}$ 
 and  coincides with the sum in $\mathbb{Z}$:

\vskip 0.1cm
\centerline{$
   0+0 = 0; \quad 1 +0= 1;\quad   -1 + 0= -1; \quad  -1 + 1 = 0;\quad   1 + (-1)= 0 \, .$}
\medskip
\noindent{\bf \em A compact representation of binary cyclotomic polynomials.}  

\begin{theorem}\label{th:complete characterization of binary cyclotomic polynomials}
 Let  $p  < q$ be prime  numbers.  Consider  the alphabet
 $\mathcal{A}=\{-1,0,1\}$ and    the left circular permutation  $\sigma$ defined in \eqref{eq: def of circular permutation}. 
 With  the  $(p-1)$ words  $ \boldsymbol{d}_0,  \boldsymbol{d}_1, \ldots  \boldsymbol{d}_{p-2} \in \mathcal{ A}^p$, 
 
 \vskip 0.1 cm
 \centerline{ 
$ 
 \boldsymbol{d}_{0} = 1 (-1)0 \cdots 0,    \qquad  \quad  \boldsymbol{d}_i = \sigma^{q} ( \boldsymbol{d}_{i-1}) = \sigma^{iq}( \boldsymbol{d}_0), \quad \hbox{for $ i \ge 1$} ,
$}
\vskip 0.1cm
\noindent
define the  $(p-1)$ words $\boldsymbol{\omega}_0, \boldsymbol{\omega}_1, \ldots, \boldsymbol{\omega}_{p-2} $, 
\begin{equation}\label{eq: definitions of omega_i}
\boldsymbol{\omega}_0 =  \boldsymbol{d}_0, \qquad \boldsymbol{\omega}_i = \boldsymbol{\omega}_{i-1} + \boldsymbol{d}_i, \quad \hbox{ for $i \ge 1$. } 
\end{equation}
Then, the following holds: 
\begin{enumerate}%[(a)]
\item\label{item thm1 a} The words  $\boldsymbol{\omega}_0, \boldsymbol{\omega}_1, \ldots \boldsymbol{\omega}_{p-2}$  belong to $ \mathcal{ A}^p $: each  addition  $\boldsymbol{\omega}_{i-1}+\boldsymbol{d}_{i}$  is internal  in $\mathcal{A}$. These $(p-1)$ words  only depend on the pair $(p, r = q \!\!\pmod  p)$.

\item  Let  $s=\lfloor q/p\rfloor$. The  cyclotomic word 
$\boldsymbol{a}_{pq}$ 
coincides  with the prefix of length $\varphi(pq)+1$ of the word 
\begin{equation} \label{bpq} \boldsymbol{b}_{pq} = \boldsymbol{\omega}_0^s \cdot \boldsymbol{\omega}_0^{r/p}\cdot \boldsymbol{\omega}_1^s \cdot \boldsymbol{\omega}_1^{r/p}\cdot \ldots\cdot  \boldsymbol{\omega}_{p-2}^s\cdot \boldsymbol{\omega}_{p-2}^{r/p}\in \mathcal{A}^{(p-1)q},
\end{equation}
 and is written  in fractional power notation as 
 \[
    \boldsymbol{a}_{pq}= \boldsymbol{\omega}_{0}^{q/p}
    \cdot  \boldsymbol{\omega}_{1}^{q/p} \cdots \,\boldsymbol{\omega}_{{p-3}}^{q/p}\cdot  \boldsymbol{\omega}_{{p-2}}^{(q-p+2)/p}.
  \]
  \end{enumerate}
\end{theorem}

\medskip
\noindent
{\bf \em  The  algorithm BCW.}  
Given two positive primes $p$ and $q$ with $p<q$, together with the quotient $s=\lfloor q/p\rfloor$  and the remainder $r=q \pmod p$ %of the division 
of $q$ by $p$,  the BCW algorithm proceeds in  three main steps. The first two steps define the precomputation phase  and only depend on the pair $(p, r= q \pmod p)$, whereas the third step performs the computation itself and depends on the pair $(q, p)$.

\smallskip 
\noindent {\bf \em Precomputation phase.} It  has two steps and computes
\\$(i)$  the words $\boldsymbol{d}_i$ for $i\in \eint{0}{p-2}$ defined in \eqref{eq: definitions of omega_i},
\\ $(ii)$   the words $\boldsymbol{\omega}_{i}$ defined in \eqref{eq: definitions of omega_i}
and their prefixes
    $\boldsymbol{\omega}_{i}^{r}$ for $i\in \eint{0}{p-2}$.

\smallskip 
\noindent {\bf \em  Computation phase.} It  computes the word $\boldsymbol{b}_{pq}$ defined in \eqref{bpq}    with  concatening   powers of  $\boldsymbol{\omega}_0,\ldots, \boldsymbol{\omega}_{p-2}$. It   depends on $s$. The cyclotomic word $\boldsymbol{a}_{pq}$ is obtained from 
$\boldsymbol{b}_{pq}$ by deleting its suffix of length $p-2$.

\medskip

The BCW algorithm only performs the  operations  described at the beginning of this section, 
on  words of length at most $p$, each one  having a cost  $\Theta(p)$. The precomputation phase performs

    -- cyclic permutations of any order of a word of length $p$;
    
      -- truncations of a suffix of length $\ell<p$; %has cost %at most $p$; 
       
  -- internal additions between  two words of length $p$. %has cost $p$.%has cost $p$.

\smallskip
\noindent{}
The computation phase performs   concatenations of a word of length at most $p$ (at the end of  an already existing word of any length); 

\medskip 
\noindent 
 Theorem 2   summarizes the analysis of the complexity  of the BCW algorithm. It will be proven in Section \ref{sec:algo}.

\begin{theorem} \label{th:cost analysis of the algorithm}
Suppose that  the  two prime numbers $p<q$ are given together with the quotient 
{$s=\lfloor q/p\rfloor$ and the remainder $r=q\!\!\pmod p $}. Then,  the BCW  algorithm  computes  the cyclotomic word $\boldsymbol{a}_{pq}$ in  $\Theta(pq)$  operations on words: 
\begin{itemize} 
\item[$(a)$]  The precomputation phase  only depends on the pair $(p,r=q \pmod p)$ and its cost is $\Theta(p^2)$. 
\item[$(b)$] The  computation phase performs $s(p-2)$ concatenations  of words of length $p$, and  its cost is  $\Theta (pq)$. 
 \end{itemize}
 \end{theorem}

\medskip

\section{A  compact representation for  binary cyclotomic  polynomials}\label{section:binary cyclotomic polynomials}

This section is devoted to  the proof of  Theorem 1 and is organized into three main parts. We first  consider  the  particular case $p=2$,   then we describe the general case 
$p>2$, and we  finally prove two auxiliary results (Proposition 1 and Lemma 1) that have been used in the proof of the general case.

\medskip 
\noindent {\bf \em  Particular case $p=2$. }
The polynomial $\Phi_{2q}$ satisfies the relation
$\Phi_{2q}(x)=\Phi_{q}(-x)$. Since $\Phi_q(x)=1+x+\dots+x^{q-1}$, it turns out that  $\Phi_{2q}$ is a polynomial of degree $q-1$ whose nonzero coefficients take  the values  $1$ and $-1$ alternatively. 
The equality  $m +1 =\varphi(2 q) + 1 =q $ holds and entails $\lfloor (m+1)/q \rfloor = 1$.    
There is only one word $\boldsymbol{\omega}_0=1(-1)$ which has to be concatenated with itself $\lfloor q/2\rfloor $ times. The resulting word is \[\boldsymbol{\omega}_0^{\lfloor q/2\rfloor}=1(-1)1(-1)\cdots 1(-1),\]
which coincides with the vector $\boldsymbol{a}_{2 q}$ of coefficients of $\Phi_{2q}$.
This ends the proof for the case $p=2$.

 \medskip 
\noindent {\bf \em  General case $p > 2$. }
We let  $m=\varphi(pq)$.  We  consider the set of words 
$\mathcal{A}^{m+1}$ as embedded in $\mathbb{Q}^{m+1}$ because we only deal with  internal additions between words. This fact  that is stated as the assertion $(\ref{item thm1 a})$ of Theorem \ref{th:complete characterization of binary cyclotomic polynomials} is proven in  Lemma \ref{lemma:suma de los di no se sale de A}.  We  then do not distinguish between words or vectors,  and the notation 
$\boldsymbol{v}=v_{0}\cdots v_{m}$ denotes both a word $\boldsymbol{v}\in \mathcal{A}^{m+1}$ and/or a vector of $\mathbb{Q}^{m+1}$.  We write  $v_{i,j}$ to denote  the $j$-th symbol  (or coordinate)  of a word 
$\boldsymbol{v}_i$. 

\smallskip  
The proof  is described along three main steps. In Step 0, we   deal with linear algebra  over the vector space ${\mathbb Q}^{m+1}$,  introduce the shift-matrix $S$ and  recall its main properties. Then  Step 1 transfers polynomial identities into linear algebra  equations on    ${\mathbb Q}^{m+1}$.  Finally,  Step 2   adopts the point of view of words.

\medskip
\noindent
 {\bf \em  Step 0. Starting with linear algebra.} 
 We consider  the vector space  $\mathbb{Q}^{m+1}$, and  the   linear map
 \[ \  S: \mathbb{Q}^{m+1}  \rightarrow  \mathbb{Q}^{m+1}  \qquad 
S  \boldsymbol{v}  = 0v_{0} \ldots v_{m-1} \quad \hbox{ for any } 
\boldsymbol{v}=v_{0}\ldots  v_{m}\in  \mathbb{Q}^{m+1} .\] 
We denote by $ \boldsymbol{e}_{0},\boldsymbol{e}_{1},\dots, \boldsymbol{e}_{m}$ the vectors of the canonical basis of  $\mathbb{Q}^{m+1}$: the word  $\boldsymbol{e}_{i}$  is the word  whose symbols  satisfy  $e_{i,i}=1$ and $e_{i,j}=0$ for $j\neq i$. 
 We deal with  the  matrix  associated with the linear map $S$ in the canonical basis that is also denoted by $S$ and called the shift-matrix.  
This is a nilpotent  matrix   of size $(m+1)\times(m+1)$ having  $1$-s along the main lower subdiagonal and $0$-s everywhere else, that satisfies %Notice that
\[S^{m+1}=0 \in \mathbb{Q}^{(m+1)\times (m+1)} \quad \hbox{ and }\quad S^{\ell} \boldsymbol{e}_0=\boldsymbol{e}_{\ell}, \hbox{ for any } \ell\in \eint{0}{m}.
\]
 With a polynomial $f(x)=a_0+a_1x+a_2x^2+\dots+ a_m x^m $,  we associate  the  matrix $ f(S) = a_0 {\rm Id}+a_1S +a_2S^2+\dots +a_m S^m$ (here ${\rm Id}$ is the identity matrix)  together with  the vector of coefficients $\boldsymbol{a}=a_0a_1\dots a_{m}\in \mathbb{Q}^{m+1}$. The following relation holds
 \begin{equation}\label{eq:a=f(s)}
\boldsymbol{a}= a_0S^0\boldsymbol{e}_0+a_1S^1\boldsymbol{e}_0+\dots +a_m S^m\boldsymbol{e}_0=f(S) \boldsymbol{e}_0.
\end{equation}
Moreover, $f(S)$ is invertible if and only if   $a_0\neq 0$. 

\medskip

\noindent {\bf \em Step 1. From polynomial identities to linear algebra equations.}
Our starting point is the polynomial identity \eqref{eq:Phipn from n}. When applied to the case when  $p$ and $q$ are distinct primes, it writes as
 \begin{equation}\label{eq:identidad de partida}
 {(1-x^{q})} \Phi_{pq}(x)={(1-x)}\Phi_{q}(x^{p}). 
 \end{equation}
   We now  specialize the previous identity in  the shift-matrix $S$ of size $(m+1)\times (m+1)$, apply it   to the vector  $\boldsymbol{e}_0$, and  obtain 
 \begin{align*}
   ( {\rm Id}-S^{q})\Phi_{pq}(S)\boldsymbol{e}_{0}= ({\rm Id}-S)\Phi_{q}(S^{p})\boldsymbol{e}_{0}.
 \end{align*}
 Relation \eqref{eq:a=f(s)} entails  the equality \begin{align} \label{eq:apq}
  \boldsymbol{a}_{pq} = \Phi_{pq}(S)\boldsymbol{e}_{0}=({\rm Id}-S^{q} )^{-1}({\rm Id}-S)\Phi_{q}(S^{p})\boldsymbol{e}_{0}.
 \end{align}
 Since $S$ is a nilpotent matrix, the inverse matrix $({\rm Id}-S^{q})^{-1}$ equals  
 \begin{equation} \label{inv} 
   \Big( {\rm Id}-S^{q}\Big)^{-1}=\sum_{i=0}^{\lfloor (m+1) /q \rfloor} S^{iq}.
 \end{equation}
 Next,  for $p>2$, the  following holds
\[
 m +1 = q ({p-2}) +  q -  (p-1) + 1, \quad  0 \leq q -  (p-1) + 1 < q \, ,
 \]
  and entails the equality $\lfloor (m+1)/q \rfloor = {p-2}$. Following \eqref{eq:apq} and \eqref{inv}, the  vector $\boldsymbol{a}_{pq}$ satisfies 
  \begin{equation}\label{eq:computation of apq via e0}
  \boldsymbol{a}_{pq} = \Phi_{pq}(S) \boldsymbol{e}_{0} =\sum_{i=0}^{ p-2} S^{iq}\boldsymbol{c}_{pq}\quad \hbox{ with } \quad \boldsymbol{c}_{pq}=
  ({\rm Id}-S)\Phi_q(S^{p})\boldsymbol{e}_{0}.
\end{equation}
The relation $S^{\ell}\boldsymbol{e}_{0} = \boldsymbol{e}_{\ell }$  holds for  integer $\ell \in \eint{0}{m}$  and yields
 \[
   \Phi_{q}(S^p)\boldsymbol{e}_0 = \sum_{j = 0}^{\lfloor \frac{m+1}{p} \rfloor} \boldsymbol{e}_{j p }.
 \]  
 Then the vector  $\boldsymbol{c}_{pq} \in \mathbb{Q}^{m+1} $ defined in \eqref{eq:computation of apq via e0} is %the element of $$ given by
\begin{align}\label{eq: cpq}
   \boldsymbol{c}_{pq} & =  \sum_{j = 0}^{\lfloor \frac{m+1}{p} \rfloor} (\boldsymbol{e}_{j p  } - \boldsymbol{e}_{j p +1})=\underset{p}{\underbrace{1 -10 \cdots 0}}\cdots   \underset{p}{\underbrace{1 -10 \cdots 0}} \, \underset{ m+1 \pmod p}{\underbrace{1 -1 0 \cdots 0}}.
\end{align}

\medskip
\noindent
{\bf \em  
Step 2. From vectors to words.}
 With \eqref{eq: cpq}, we now view  the vector $\boldsymbol{c}_{pq}$ 
 as a word of  $\mathcal{A}^{m+1}$ that   is written  as a fractional power of the word $\boldsymbol{d}_0\in \mathcal{A}^{p}$ defined in the statement of Theorem \ref{th:complete characterization of binary cyclotomic polynomials}: 
\begin{equation}\label{eq: cpq as a power}
\boldsymbol{c}_{pq}= \boldsymbol{d}_{0}^{(m+1)/p}\in \mathcal{A}^{m+1}.
\end{equation}
 We now use  the next proposition (Proposition 1) that  states that such a fractional power can be written in terms of the transforms $\boldsymbol{d}_i$ of $\boldsymbol{d}_0$ via cyclic permutations. The proof of Proposition 1 will be found at the end of this section. 

\begin{proposition}\label{prop:resultados previos sobre palabras}
 The fractional power $ \boldsymbol{d}_{0}^{(m+1)/p}$  of $\boldsymbol{d}_{0}$ is written as a concatenation of fractional powers of  the  cyclic transforms $\boldsymbol{d}_i=\sigma^{iq}( \boldsymbol{d}_{0})$:  
\[
\boldsymbol{c}_{pq}= \boldsymbol{d}_{0}^{(m+1)/p} = \boldsymbol{d}_{0}^{q/p}\cdot \boldsymbol{d}_{1}^{q/p} \cdots  \boldsymbol{d}_{p-3}^{q/p}  \cdot \boldsymbol{d}_{p-2}^{(q-p+2)/p}.
\]
\end{proposition}

 We then return  to the proof. In \eqref{eq:computation of apq via e0}, 
 we  now view, for  each $i \in  \eint{0}{p-2} $,    $S^{iq} \boldsymbol{c}_{pq} $ 
 as a word in $\mathcal{A}^{m+1}$, written as
 \[
   S^{iq}\boldsymbol{c}_{pq} = 0^{iq} \cdot \boldsymbol{d}_{0}^{q/p} \cdot \boldsymbol{d}_{1}^{q/p}\, \cdots \, \boldsymbol{d}_{p-2 - i}^{(q-p+2)/p}, 
 \]
 where $0^{iq}$ is the word obtained with the concatenation of  $iq$  consecutive zeros. From \eqref{eq:computation of apq via e0}, the
 cyclotomic word $\boldsymbol{a}_{pq}$ is expressed in terms of $\boldsymbol{c}_{pq}$ as
 \begin{equation} \label{sum}
 \boldsymbol{a}_{pq}=\sum_{i=0}^{ p-2} S^{iq}\boldsymbol{c}_{pq}=\sum_{i=0}^{p-2} 0^{iq}\boldsymbol{d}_{0}^{q/p}\cdot \boldsymbol{d}_{1}^{q/p}  
  \cdots   \boldsymbol{d}_{p-2 - i}^{(q-p+2)/p}.
\end{equation} 
 We now explain how the sums $ \boldsymbol{\omega}_i$ that intervene  in Theorem  \ref{th:complete characterization of binary cyclotomic polynomials} appear. We consider the  $(p-2)$  following vectors of $\mathbb Q^{q}$   \[\boldsymbol{f}_0=\boldsymbol{d}_{0}^{q/p},\quad   \boldsymbol{f}_1=\boldsymbol{d}_{1}^{q/p},\quad   \boldsymbol{f}_2=\boldsymbol{d}_{2}^{q/p},\ldots,\quad  \boldsymbol{f}_{p-3}=\boldsymbol{d}_{p-3}^{q/p},\] together  with the $(p-2)$  following vectors of $\mathbb{Q}^{q-p+2}$:
\begin{equation} \label{last}\boldsymbol{g}_{0}=\boldsymbol{d}_{0}^{(q-p+2)/p},\  \boldsymbol{g}_{1}=\boldsymbol{d}_{1}^{(q-p+2)/p},\  \ldots, \ \boldsymbol{g}_{p-2}=\boldsymbol{d}_{p-2}^{(q-p+2)/p}.
\end{equation}
 We  use the  expression of the vector  $ \boldsymbol{a}_{pq}$ in \eqref{sum}  as a sum  that uses   $(p-1)$ lines,  indiced from $0$ to $p-2$, each line for each term of the sum, 
\begin{align*}
    \boldsymbol{a}_{pq} = & \,  \boldsymbol{f}_0\cdot \boldsymbol{f}_1\cdot\boldsymbol{f}_2 \cdot\boldsymbol{f}_3\cdots \boldsymbol{f}_{p-3}\cdot\boldsymbol{g}_{p-2} 
    \\
     + & \, 0^{q}\cdot\boldsymbol{f}_0\cdot\boldsymbol{f}_1 \cdot\boldsymbol{f}_2\cdots \boldsymbol{f}_{p-4}\cdot\boldsymbol{g}_{p-3}
     \\
    +& \, 0^{q}\cdot 0^{q}\cdot\boldsymbol{f}_0\cdot\boldsymbol{f}_1\cdots \boldsymbol{f}_{p-5}\cdot\boldsymbol{g}_{p-4}
    \\
     \vdots &  
     \\
     + & \;  0^{q}\cdot 0^{q}\cdot 0^{q}\cdot 0^{q}\cdots  \ \  0^{q}\ \ \cdot \ \ \boldsymbol{g}_{0}.
\end{align*}
Each line exactly  follows the same pattern: it is formed with $(p-2)$ blocks of length $q$, followed with a block of length $(q-p+2)$.  The  line of index $i$  begins with  $i$ blocks of $q$ zeroes,  that perform a	 {\em shift} between blocks.   Then, when we read the result of the  previous sum  by columns, indexed from $0$ to $p-1$,  the $j$-th column contains  (for $0\le j <p-2$) the element $\boldsymbol{f}_0+ \cdots + \boldsymbol{f}_j$, whereas  the last column equals the sum of elements defined in \eqref{last}.  Provided that all the sums be internal,  this implies the equality 
 \begin{align*}
  \boldsymbol{a}_{pq}  = \boldsymbol{d}_{0}^{q/p}\cdot(\boldsymbol{d}_{0}+\boldsymbol{d}_{1})^{q/p}\cdots (\boldsymbol{d}_{0} +  \cdots + \boldsymbol{d}_{{p-3}})^{q/p} \cdot(\boldsymbol{d}_{0} +  \cdots 
     +\boldsymbol{d}_{p-2})^{(q-p+2)/p} \, ,
 \end{align*}   
  and  introduces  the words $  \boldsymbol{\omega}_i$ involved in Theorem \ref{th:complete characterization of binary cyclotomic polynomials}. 
 
\medskip 
The next lemma shows that  the computation of the words $\boldsymbol{\omega}_i=\boldsymbol{d}_0+\cdots +\boldsymbol{d}_{i}$ defined in Theorem 1  indeed  involves internal operations in $\mathcal{A}$.  It will be proven at the end of the section.  It provides  the proof of  Statement $(\ref{item thm1 a})$ of Theorem \ref{th:complete characterization of binary cyclotomic polynomials}.

\begin{lemma}\label{lemma:suma de los di no se sale de A} Under the assumptions and notations of Theorem
\ref{th:complete characterization of binary cyclotomic polynomials}, the words $\boldsymbol{\omega}_{i}$ belong to $ \mathcal{A}^p$ for each $i\in \eint{0}{p-2}$. 
\end{lemma}
Now, with Lemma 1,  the cyclotomic word $\boldsymbol{a}_{pq}$ is thus expressed as a concatenation of fractional powers, 
  \[
   \boldsymbol{a}_{pq} =  \boldsymbol{\omega}_{0}^{q/p}\cdot \boldsymbol{\omega}_{1}^{q/p} \,\cdots \, \boldsymbol{\omega}_{p-3}^{q/p} \cdot \boldsymbol{\omega}_{p-2}^{(q-p+2)/p}\, , 
 \]
  it  has length $m+1$ and  coincides with the prefix of length $m+1$ of the word \[
 \boldsymbol{b}_{pq}=  \boldsymbol{\omega}_{0}^{q/p}\cdot \boldsymbol{\omega}_{1}^{q/p} \,\cdots \, \boldsymbol{\omega}_{p-3}^{q/p} \cdot \boldsymbol{\omega}_{p-2}^{q/p}.\]
  This ends the proof of Theorem 1.

  \medskip We now provide the proof of the two results that we  have used.

 \medskip
 \noindent{\bf \em Proof of  Proposition \ref{prop:resultados previos sobre palabras}.}
The first step  of the proof is based on a simple fact which holds for any positive integers $m$, $p$ and $q$,  with $m>p$ and $q \in \eint{p+1}{m}$, and any   word $\boldsymbol{d}_{0}\in \mathcal{A}^{p}$. The fact is the following: if  $r=q\pmod p$, then
\[\boldsymbol{d}^{(m+1)/p}_{0}=\boldsymbol{d}^{q/p}_{0}\cdot\boldsymbol{d}_{1}^{(m+1-q)/p}\quad \hbox{ with }\quad \boldsymbol{d}_{1} = \sigma^{q}(\boldsymbol{d}_{0})= \sigma^{r}(\boldsymbol{d}_{0})\in \mathcal{A}^p.\]
Here we prove it. First, it is clear that
\[\boldsymbol{d}_{0}^{(m+1)/p}=\boldsymbol{d}^{q/p}_{0}\cdot\boldsymbol{v} \quad \hbox{ with }\quad \boldsymbol{v}\in \mathcal{A}^{m+1-q}.\] 

Let $d_j$ and ${v}_j$ denote the $j$th coordinates of   $\boldsymbol{d}_0$ and $\boldsymbol{v}$, respectively. It is clear that  ${v}_j$ equals the $(j+r)\!\!\pmod p$ coordinate of $\boldsymbol{d}_0$. 
     This implies that
     $ \boldsymbol{v}=\boldsymbol{d}_{1}^{(m+1-q)/p}$, 
     with $\boldsymbol{d}_{1}   = \sigma^{q}(\boldsymbol{d}_{0})= \sigma^{r}(\boldsymbol{d}_{0}) = d_{r + 1} \cdots d_{p} d_{0} \cdots d_{r}$.  This is what we wanted to prove.

For the particular choice $m=\varphi(pq)$, it turns out that $\lfloor (m+1)/q \rfloor=p-2$ and $m+1=q-p+2 \pmod q$ (recall that $p>2$).

 Now, we apply an inductive argument.          
      For $ i = 0, \ldots,  p-2 $, set $\boldsymbol{d}_{i} = \sigma^{r}(\boldsymbol{d}_{i-1})$.
  Then 
  \[
  \boldsymbol{d}_{0}^{(m+1)/p} = \boldsymbol{d}_{0}^{q/p} \cdot \boldsymbol{d}_{1}^{(m+1-q)/p}= \boldsymbol{d}_{0}^{q/p} \cdot  \boldsymbol{d}_{1}^{q/p}\cdots \,\boldsymbol{d}_{p -3}^{q/p}\cdot \boldsymbol{d}_{p-2}^{(q-p+2)/p}.      
  \]

 \medskip
 \noindent{\bf\em Proof of   Lemma \ref{prop:resultados previos sobre palabras}.}
Fix $i\in \eint{1}{p-2}$ and  let  $j$ be in the interval $ \eint{0}{p-1}$.   
 With $r=q\pmod p$, the relation \[\boldsymbol{d}_{k}=\sigma^{r}(\boldsymbol{d}_{k-1})=\sigma^{kr}(\boldsymbol{d}_{0})\] shows that 
\[
{d}_{k,j}=
    \begin{cases}
    1\phantom{-}& \hbox{ if  } j+kr=0 \pmod p,\\
    -1 &\hbox{ if  } j+kr=1 \pmod p,\\
    0 & \hbox{ otherwise }.
    \end{cases}
\]
Now, we proceed with an inductive argument.
The statement of the lemma is  clear for $\boldsymbol{\omega}_0=\boldsymbol{d}_0$. 
Suppose that $\boldsymbol{\omega}_{i-1}\in \mathcal{A}^{p}$. 
From the definition of words  $\boldsymbol{\omega}_i$, given in \eqref{eq: definitions of omega_i},  it is clear that \[\boldsymbol{\omega}_{i-1}=\boldsymbol{d}_0+\boldsymbol{d}_1+\dots+\boldsymbol{d}_{i-1}.\] 
Hence, if ${\omega}_{i-1,j}=1$, then  $j+kr=0 \pmod p$ for some $k\in \eint{0}{i-1}$; and similarly, if ${\omega}_{i-1,j}=-1$, then $j+kr=1 \pmod p$ for some $k\in \eint{0}{i-1}$. Remark that the converse might not be true because there might be cancellations between $1$ and $-1$.

Since $\gcd(p,r)=1$, we have $k r\neq i r\pmod p$ for any $k\in\eint{0}{i-1}$. It follows  that  there is no $j\in \eint{0}{p-1}$ such that ${\omega}_{i-1,j}={d}_{i,j}=1$ or ${\omega}_{i-1,j}={d}_{i,j}=-1$. Then,  the sum
$\boldsymbol{\omega}_i=\boldsymbol{\omega}_{i-1}+\boldsymbol{d}_i\in \mathcal{A}^p$, that is, the  sums of two $1$ or two $-1$ never occur.

\section {Algorithms} \label{sec:algo}
 
This section  considers algorithms, and  analyses their complexities. It is organised into three main parts. It begins with  the BCW algorithm,  then explains how it may be used to compute tables of binary cyclotomic polynomials, and finally  compares the BCW algorithms with  other algorithms that have been previously proposed.

\medskip
\noindent{\bf \em Analysis of the BCW  algorithm. } We first analyse the BCW algorithm designed in Section 2.  This will provide the proof of  Theorem \ref{th:cost analysis of the algorithm}.  The analysis of the BCW algorithm 
follows from  the  next two lemmas.

\begin{lemma}\label{lemma: precomputation step}{\rm [Precomputation step].}  
 The output ${\mathcal O}_{p, r}$  of the precomputation step  on the input $(p, r= q \pmod p)$ contains the words $\boldsymbol{\omega}_{i}$ together with their prefixes $\boldsymbol{\omega}_{i}^{r/p} $, for $i= 0, \ldots, p-2$. It  is computed  with  $\Theta(p)$  operations on words of   $\mathcal{A}^p$.% and the arithmetic operation $r=q\pmod p$.   
The total cost  for computing  ${\mathcal O}_{p, r}$ is $\Theta(p^2)$.   
\end{lemma}

\begin{proof} In Lemma \ref{lemma:suma de los di no se sale de A}, we already proved that all the additions between words $\boldsymbol{\omega}_i$ and words $\boldsymbol{d}_i$ are internal to the alphabet $\mathcal{A}$.  Given $r=q \pmod p$, the $p-2$ words $\boldsymbol{d}_{i}$ are obtained by applying $p-2$ cyclic permutations of order $r$ to a word of length $p$.  The words $\boldsymbol{\omega}_{i}$ are obtained with  $\Theta(p)$ internal additions of words in $\mathcal{A}^p$.  The computation of the prefixes $\boldsymbol{\omega}_{i}^{r/p} $ only involves $p-2$ truncations. All the words involved has length $p$. Thus, the total cost is $\Theta(p^2)$.

\end{proof}

\begin{lemma} \label{lemma: concatenation step} {\rm [Concatenation step].} The concatenation step computes $\boldsymbol{a}_{pq}$ from the output  ${\mathcal O}_{p, r}$ of the precomputation step in $\Theta(q)$  words operations (mainly concatenations).
The total cost of the concatenation step is $\Theta(pq)$. 
 \end{lemma}

\begin{proof}  We consider   the quotient $s$  and the remainder $r$ of the division of $q$ by $p$. 
The  word $\boldsymbol{b}_{pq}$ is obtained by  concatenating the words  $\boldsymbol{\omega}_{i}\in \mathcal{A}^p$ or $\boldsymbol{\omega}_{i}^{r/q}\in \mathcal{A}^r$  according to the order prescribed in Theorem \ref{th:complete characterization of binary cyclotomic polynomials}. There are $\Theta(sp)=\Theta(q)$ such concatenations of words of length at most $p$.  
Finally, to  obtain $\boldsymbol{a}_{pq}$ from $\boldsymbol{b}_{pq}$, it suffices to delete the last $p-2$ symbols. 
The total cost of this step is $\Theta(pq)$.  
\end{proof}

% \begin{remark}
% Given a prime number $p$ and any $r\in \eint{1}{p-1}$, it is possible to produce the  $p-1$ words of length $p$ of the precomputation step, $\boldsymbol{\omega}_{0},\ldots, \boldsymbol{\omega}_{p-2}$ together with their prefixes of length $r$. Suppose that $q_{1} < q_{2}$ are primes numbers such that $r=q_{1} \pmod p=q_{2} \pmod p$. Let $s_1$ and $s_2$ be the corresponding quotients of the division by $p$. Then,  the  cyclotomic word $\boldsymbol{a}_{p q_{2}} $ can be computed from $\boldsymbol{a}_{p q_{1}}$ just inserting $s_2-s_1$ times the words $\boldsymbol{\omega_i}$. We return  to this remark in Section XX
 %\end{remark}

\medskip
\noindent{\bf \em  Computing tables of binary cyclotomic polynomials with the BCW algorithm.}
We denote by 
${\mathcal O}_p$  the total precomputation output  relative to a fixed  prime $p$, that is the union of  the precomputation outputs  relative this fixed prime $p$ and  all  possible  values of $r \in \eint{1}{p-1}$.  The cost of building  the total output  ${\mathcal O}_p$ is $\Theta(p^3)$. 

\smallskip
\noindent We discuss the cost of  computing   a list of  polynomials  $\Phi_{p_0 q}$ for a fixed  prime $p_0$ and  a number $t = p_0^\beta$  of primes $q$ that  all satisfy $q \le p_0^\alpha$ (with $\alpha >1$).

\smallskip
\noindent 
 {\bf \em Case $(a)$.}  Consider first the case when we wish to compute  the polynomials $\Phi_{p_0 q}$ for a number  $t= p_0^\beta $  of primes  $q$ that  satisfy  $q \le p_0^\alpha$ (with $\alpha \ge 1$)  and $q \pmod p = r$. We then precompute  only the output ${\mathcal O}_{p_0, r}$, with a cost $\Theta(p_0^2)$, then we perform $t$ computations of cost $p_0^{1+ \alpha}$. The  total cost is  
$$   \Theta(p_0^2) +  \Theta (p_0^{1+ \alpha+\beta})= \Theta (p_0^{1+ \beta+ \alpha}) \, .$$  In this case, the cost of the precomputation is always smaller than the cost of the computation.

\smallskip
\noindent 
{\bf \em Case $(b)$.}   Consider now the case when we wish to compute  the polynomials  $\Phi_{p_0 q}$ for a number  $t= p_0^\beta $ of  primes  $q$ that all satisfy  $q \le p_0^\alpha$, with $\alpha \ge 1$  and various values of $q \pmod p$.  We have to compute the total  precomputation output $\mathcal{O}_{p_0}$ in $\Theta(p_0^3)$ steps,   then  we perform $t$ computations of cost $p_0^{1+ \alpha}$
The  total cost is  
$$   \Theta(p_0^3) +  \Theta (p_0^{1+ \alpha+\beta})= \Theta(p_0^{1 + \max((\alpha +\beta) , 2)})\, .$$
The cost of the precomputation is larger than the cost of the computation,  when    $\alpha + \beta$ is smaller than 2.

\smallskip
\noindent 
{\bf \em Case $(c)$.} Consider finally the case  when we wish to compute  the polynomials  $\Phi_{p_0 q}$ for  any prime $q \le p_0^\alpha$, with $\alpha \ge 1$  with various values of $q \pmod p$. We assume that the list of primes $q \le p_0^\alpha$ is already built. There are $ t = \Theta (1/ \log p_0) p_0^\alpha$  such primes $q$.  
We have to compute the total output  $\mathcal{O}_{p_0}$ in $\Theta(p_0^3)$ steps,   then  we perform $t= O (p_0^\alpha)$ computations of cost $p_0^{1+ \alpha}$
The  total cost is  
$$   \Theta(p_0^3) +  \Theta (p_0^{1+ 2\alpha})= \Theta (p_0^{1+ 2 \alpha})\, . $$
The cost of the precomputation is always  smaller than the cost of the computation.

\medskip
\noindent{\bf \em Comparison with  other algorithms.} %\label{section: perfomance and comparison}
We now  briefly describe algorithms  that have been previously designed to  compute  cyclotomic polynomials. Even when they are designed in the case of a general cyclotomic polynomial,  we  only  analyse their complexity  in the binary case, and compare them to the  complexity of the BCW algorithm.  
In the following, $M(\ell)$  denotes the cost of multiplying two integers of bit-length at most $\ell$.   With fast-multiplication algorithms, $M(\ell)$   is $O(\ell\log \ell \log\log\ell)$.

\medskip
$(i)$ The first algorithm is based on a formula  due to  Lenstra-Lam-Leung \cite{LaLe99,Lenstra79},  and is 
only valid  in the binary case.  
 With this formula, 
the $j$-th coefficient of $\boldsymbol{a}_{pq}$ is  computed via a 
solution $(x,y)$ of an  equation of {the} type $j=xp+yq$ with $|x|<q$ and $|y|<p$. The cost of  computing  one single coefficient  is thus  $O(M(\log q))$. The  total cost for  the whole vector $\boldsymbol{a}_{pq}$  is $\Theta(pq) M(\log q)$. 

\smallskip
$(ii)$  The  other two algorithms are due to 
 Arnold and Monagan and are described in \cite{ArMo11}. They   compute the  cyclotomic vector  in the  case of a general  order $n$, but we describe them for a binary order $n = pq$. They do not perform any  
polynomial divisions.

\smallskip
$(ii) (a)$ The first algorithm   is  a recursive algorithm, called the sparse power series algorithm  (SPS algorithm, for short) because it expresses cyclotomic polynomials  
as products of sparse power series. 
In the  binary case, the recursion is made with  $\Theta(pq)$ additions between integers less than (but close to) $\varphi(pq)$. 
Then, its complexity  is  $\Theta(pq)\log q$.

\smallskip
$(ii) (b)$ The second  algorithm   is called the Big Prime algorithm (BP, for short). It is  adapted to the case when the order  $n$  has  a big prime as a factor. Here,  in the binary case, the big prime is $q$. 
The algorithm  is based on Identity \eqref{eq:tn-1 equal to product of cyclotomic} and  performs a precomputation step which involves $\Theta(p)$ computations of residues  modulo $p$ of integers of size $O(\log q))$. The cost of this step is then $\Theta(p) M(\log q))$. The whole cyclotomic word $\boldsymbol{a}_{pq}$ is obtained via a recursive step which  computes $\varphi(pq)$ residues modulo $p$ of integers of size $O(\log q)$. Thus, the cost of  computing the whole vector $\boldsymbol{a}_{pq}$ is $\Theta(pq)M(\log q)$.

\medskip
The previous discussion shows that the   existing algorithms  are all of  complexity  $O(pq) \cdot E(p, q)$ where the factor $E(p, q)$ is polynomial in $\log q$.    The complexity of the BCW algorithm  does not contain such a factor $E(p, q)$.

\section{Conclusions and further work}
 In this conclusion, we first explain how our results may be extended to another family of polynomials, related to semi-groups. We then discuss how the representation of Theorem \ref{th:complete characterization of binary cyclotomic polynomials}
 may entail results on the length of blocks of  consecutive zeros in the cyclotomic words. The paper ends with a short discussion on the non-binary case. 

\medskip
\noindent
{\bf \em Another polynomial $F_{p,q}$.} We consider  two   numbers  $(p, q)$ with $2<p<q$,  that are coprime but  now {\em not necessarily primes}, and  define the polynomial $F_{p,q}\in \mathbb{Z}[x]$ as follows:
\[
F_{p,q}(x)=\frac{(x^{pq}-1)(x-1)}{(x^p-1)(x^q-1)}.
\]
A folklore result relates the polynomial $F_{p,q}$ to numerical semi-groups (see, \cite{Moree14} for definitions and   proofs).  
A numerical semi-group is the set 
\[
 S(p,q)=\{ap+bq\mid a,b\in \mathbb Z_{\ge 0}\}.
 \]
The polynomial $F_{p,q}$ is called the semi-group polynomial associated with $S(p,q)$  due to  the following identity
\[
F_{p,q}(x)=1+(x-1) \sum_{s\notin S(p,q)}x^s. 
\]
Clearly, $ F_{p,q}$ coincides with the cyclotomic polynomial $\Phi_{pq}$ when $p$ and $q$ are  primes. Also, if we define $F_q=1+x+\cdots +x^{q-1}$, the polynomial $F_{p,q}$ satisfies  the  identity \eqref{eq:identidad de partida} that is our starting point in the proof of Theorem \ref{th:complete characterization of binary cyclotomic polynomials}, namely
\[{(1-x^{q})} F_{p,q}(x)={(1-x)}F_{q}(x^{p}).\]
This entails that all
our results 
also hold for semi-groups polynomials $F_{p,q}$. 

\medskip
\noindent{\bf \em  Proving other properties for $\Phi_{pq}$ and $F_{p, q}$.}
The representation  provided  in Theorem  \ref{th:complete characterization of binary cyclotomic polynomials}  implies that the coefficients of $\Phi_{pq}$ belong to $\mathcal{A}$. 
A parameter of interest is the maximum gap of these polynomials, that we now define:  for a given polynomial $f(x)=b_1x^{n_1}+b_2x^{n_2}+\dots+b_{k}x^{n_k}$, with $b_{i}$ all nonzero and $n_{1}<n_{2}<\dots <n_{k}$, the maximum gap is defined as 
\[
g(f)=\max_{1\le i<k}(n_{i+1}-n_i),\qquad g(f)=0 \hbox{ if } k=1.
\]
The fact that $g(F_{p,q})=p-1$ was proven in \cite{HongLeePark12} for binary  cyclotomic polynomials $\Phi_{pq}$ and in \cite{Moree14} for semi-group polynomials. 
In \cite{Camburu16,Zhang16},  it is proven that the number of maximum gaps in binary cyclotomic polynomials is  $2\lfloor q/p\rfloor$. 

\smallskip

Our description in terms of words easily entails the inequality 
$g(F_{p,q})\ge p-1$ and  proves that there are, at least, $2\lfloor q/p\rfloor$ maximum gaps. The upper bounds   will be  recovered from our Theorem \ref{th:complete characterization of binary cyclotomic polynomials} {\em provided that} the possible cancellations $1+(-1)$  {\em be controlled} in the sums $\boldsymbol{\omega}_{i-1}+\boldsymbol{d}_{i}$.  This is the aim of a further work. 

\medskip\noindent{\bf \em Possible extensions to the non-binary case?}  Our results seem to be  specific to the binary case where the order is $n= pq$. In the case when  $n$  is squarefree with at least three prime factors, the coefficients of the cyclotomic polynomial do not any longer belong to a finite alphabet $\mathcal{A}$. It seems thus difficult to easily deal with words on this alphabet  $\mathcal{A}$. 

\medskip
\noindent{\bf Acknowledgments.} The authors wish to thank Brigitte Vall\'ee for many helpful conversations and suggestions.


\begin{thebibliography}{1}

\bibitem{ArMo11}
A.~Arnold and M.~Monagan.
\newblock Calculating cyclotomic polynomials.
\newblock {\em Math. Comp.}, 80(276):2359--2379, 2011.

\bibitem{Camburu16}
O.~Camburu, E.~Ciolan, F.~Luca, P.~Moree, and I.~Shparlinski.
\newblock Cyclotomic coefficients: gaps and jumps.
\newblock {\em J. Number Theory}, 163:211--237, 2016.

\bibitem{HongLeePark12}
H.~Hong, E.~Lee, H.~Lee, and C.~Park.
\newblock Maximum gap in (inverse) cyclotomic polynomial.
\newblock {\em J. Number Theory}, 132(10):2297--2315, 2012.

\bibitem{LaLe99}
T.Y. {Lam} and K.H. {Leung}.
\newblock On the cyclotomic polynomial $\phi\sb{pq}(x)$.
\newblock {\em Amer. Math. Mon.}, 103(7):562--564, 1996.

\bibitem{Lenstra79}
H.~W. Lenstra, Jr.
\newblock Vanishing sums of roots of unity.
\newblock In {\em Proceedings, {B}icentennial {C}ongress {W}iskundig
  {G}enootschap ({V}rije {U}niv., {A}msterdam, 1978), {P}art {II}}, volume 101
  of {\em Math. Centre Tracts}, pages 249--268. Math. Centrum, Amsterdam, 1979.

\bibitem{Moree14}
P.~Moree.
\newblock Numerical semigroups, cyclotomic polynomials, and {B}ernoulli
  numbers.
\newblock {\em Amer. Math. Mon.}, 121(10):890--902, 2014.

\bibitem{Zhang16}
B.~Zhang.
\newblock Remarks on the maximum gap in binary cyclotomic polynomials.
\newblock {\em Bull. Math. Soc. Sci. Math. Roumanie (N.S.)},
  59(107)(1):109--115, 2016.

\end{thebibliography}
\end{document}